\documentclass{amsart}
\usepackage{amssymb}
\input xy
\xyoption{all}

\newcommand{\pr}{\mathrm{pr}}

\newcommand{\C}{\mathcal C}

\newcommand{\B}{\mathcal B}

\newcommand{\N}{\mathbb N}
\newcommand{\R}{\mathbb R}

\newcommand{\F}{\mathcal F}

\newtheorem{theorem}{Theorem}

\newtheorem{lemma}{Lemma}

\newtheorem{corollary}{Corollary}

\begin{document}

\title{Games in possibility capacities with  payoff expressed by fuzzy integral}

\author{Taras Radul}

\maketitle

Institute of Mathematics, Casimirus the Great University of Bydgoszcz, Poland;
\newline
Department of Mechanics and Mathematics, Ivan Franko National University of Lviv,
Universytetska st., 1. 79000 Lviv, Ukraine.
\newline
e-mail: tarasradul@yahoo.co.uk



\begin{abstract} This paper studies non-cooperative games where players are allowed  to play their mixed non-additive strategies.    Expected payoffs are expressed by so-called fuzzy integrals: Choquet
integral, Sugeno integral and generalizations of Sugeno integral obtained by using triangular norms.  We consider the existence problem of Nash equilibrium for such games. Positive results for Sugeno integral and its generalizations are obtained. However we provide some example of a game with Choquet payoffs which have no Nash equilibrium. Such example demonstrates that fuzzy integrals based on the maximum operation are more suitable for possibility capacities then  Choquet integral which is based on the addition operation.
\end{abstract}


\section{Introduction}

The classical Nash equilibrium theory is based on fixed point theory and was developed in frames of linear convexity. The mixed strategies of a player are probability (additive) measures on a set of pure strategies. But an interest to Nash equilibria in more general frames is rapidly growing in last decades. For instance,
Aliprantis, Florenzano and Tourky \cite{AFT} work in ordered topological vector spaces, Luo \cite{L} in topological semilattices,
Vives \cite{Vi} in complete lattices.  Briec and Horvath \cite{Ch} proved  existence of Nash equilibrium point for idempotent convexity.

We can use additive measures only when we know precisely probabilities of all events considered in a game. However, it is not a case
in many modern economic models. The decision theory under uncertainty considers a model when probabilities of states are either not known or imprecisely specified. Gilboa \cite{Gil} and Schmeidler  \cite{Sch} axiomatized  expectations expressed by Choquet
integrals attached to non-additive measures called capacities (fuzzy measures), as a formal approach to decision-making under uncertainty.

Dow and Werlang \cite{DW} used this approach for two players game where belief of each player about a choice of the strategy by the other player is a convex capacity, but the players play with pure strategies. They introduced some equilibrium notion for such games and proved its existence.  This result was extended onto games with arbitrary finite number of players in \cite{EK}. Another interesting approach to the games in convex capacities with pay-off functions expressed by Choquet integrals can be find in \cite{Ma}. The authors considered finite sets of pure strategies in the above mentioned papers.

An alternative to so-called Choquet expected utility model is the qualitative decision theory.   The corresponding expected utility is expressed by Sugeno integral. This approach was widely studied in the last decade  (\cite{DP},\cite{DP1},\cite{CH1},\cite{CH}).  Sugeno integral chooses a median value of utilities which is qualitative counterpart of the averaging operation by Choquet integral.

The equilibrium notion from \cite{DW} and \cite{EK} for a game  with expected payoff function defined by Sugeno integral was considered in \cite{R4}. The sets of pure strategies are arbitrarily compacta.  Let us remark that in \cite{DW} and \cite{EK} attention was restricted to convex capacities which play an important role in Choquet expected utility theory. There are two important classes of capacities in the qualitative decision theory, namely  possibility and necessity capacities which describe optimistic and pessimistic criteria \cite{DP}. The existence of equilibrium expressed by possibility (or necessity) capacities is proved in  \cite{R4}. Since the spaces of possibility and  necessity capacities have no natural linear convex structure, some non-linear convexity is used.

Kozhan and Zarichnyi  \cite{KZ} and Glycopantis and Muir \cite{GM} considered games with Choquet payoff where players are allowed to form non-additive beliefs about opponent's decision but also to play their mixed non-additive strategies expressed by capacities.  The same approach for games with Sugeno payoff was considered in \cite{R3}. Games with strategies expressed by possibility capacities were recently considered by Hosni and Marchioni \cite{HM}. They considered payoff functions represented by Choquet integral and Sugeno integral.   Let us remark that when we consider the space of all capacities which has the greatest and the smallest elements (or the space of possibility capacities which has the greatest elements), then the existence  problem of Nash equilibrium is rather trivial. But the set of possibility capacities has no smallest element and  the set of necessity capacities has no greatest element. So, if we consider a game  where  the players play their mixed  strategies expressed by possibility capacities and the goal of each player is to minimize his expected payoff function,  existence of Nash equilibrium is not trivial for such games. (Dually, it is possible to consider games in necessity capacities and the goal to maximize expected payoff function.) We will prove existence of Nash equilibrium for games with expected payoff functions represented by fuzzy integral generated by the maximum operation and some continuous triangular norm (a partial case is the  Sugeno integral which is generated by the maximum and the minimum operations). We also provide an example of a game in possibility capacities with minimizing of  expected payoff functions represented by  Choquet integral which has no Nash equilibrium. This example demonstrates that the Choquet integral is not so suitable for  possibility capacities as it is for convex capacities (see for example \cite{DW}, \cite{EK} and \cite{Ma}).

\section{Capacities and fuzzy integrals} In what follows, all spaces are assumed to be compacta (compact Hausdorff space) except for $\R$ and all maps are assumed to be continuous. By $\F(X)$ we denote the family of all closed subsets of a compactum $X$. We shall denote the
Banach space of continuous functions on a compactum  $X$ endowed with the sup-norm by $C(X)$. For any $c\in\ R$ we shall denote the
constant function on $X$ taking the value $c$ by $c_X$. We also consider natural lattice operations $\vee$ and $\wedge$ on $C(X)$ and  its sublattices $C(X,[0,+\infty))$ and $C(X,[0,1])$.

We need the definition of capacity on a compactum $X$. We follow a terminology of \cite{NZ}.
A function $\nu:\F(X)\to [0,1]$  is called an {\it upper-semicontinuous capacity} on $X$ if the three following properties hold for each closed subsets $F$ and $G$ of $X$:

1. $\nu(X)=1$, $\nu(\emptyset)=0$,

2. if $F\subset G$, then $\nu(F)\le \nu(G)$,

3. if $\nu(F)<a$, then there exists an open set $O\supset F$ such that $\nu(B)<a$ for each compactum $B\subset O$.

If $F$ is a one-point set we use a simpler notation $\nu(a)$ instead $\nu(\{a\})$.
A capacity $\nu$ is extended in \cite{NZ} to all open subsets $U\subset X$ by the formula $\nu(U)=\sup\{\nu(K)\mid K$ is a closed subset of $X$ such that $K\subset U\}$.

It was proved in \cite{NZ} that the space $MX$ of all upper-semicontinuous  capacities on a compactum $X$ is a compactum as well, if a topology on $MX$ is defined by a subbase that consists of all sets of the form $O_-(F,a)=\{c\in MX\mid c(F)<a\}$, where $F$ is a closed subset of $X$, $a\in [0,1]$, and $O_+(U,a)=\{c\in MX\mid c(U)>a\}$, where $U$ is an open subset of $X$, $a\in [0,1]$. Since all capacities we consider here are upper-semicontinuous, in the following we call elements of $MX$ simply capacities.

A capacity $\nu\in MX$ for a compactum $X$ is called  a necessity (possibility) capacity if for each family $\{A_t\}_{t\in T}$ of closed subsets of $X$ (such that $\bigcup_{t\in T}A_t$ is a closed subset of $X$) we have $\nu(\bigcap_{t\in T}A_t)=\inf_{t\in T}\nu(A_t)$  ($\nu(\bigcup_{t\in T}A_t)=\sup_{t\in T}\nu(A_t)$). (See \cite{WK} for more details.) We denote by $M_\cap X$ ($M_\cup X$) a subspace of $MX$ consisting of all necessity (possibility) capacities. Since $X$ is compact and $\nu$ is upper-semicontinuous, $\nu\in M_\cap X$ iff $\nu$ satisfy the simpler requirement that $\nu(A\cap B)=\min\{\nu(A),\nu(B)\}$.

If $\nu$ is a capacity on a compactum $X$, then  the function $\kappa X(\nu)$, that is defined on the family $\F(X)$  by the formula $\kappa X(\nu)(F) = 1-\nu (X\setminus F)$, is a capacity as well. It is called the dual
capacity (or conjugate capacity ) to $\nu$. The mapping $\kappa X : MX \to MX$ is a homeomorphism and an involution \cite{NZ}. Moreover, $\nu$ is a necessity capacity if and only if $\kappa X(\nu)$ is a possibility capacity. This implies in particular that $\nu\in M_\cup X$ iff $\nu$ satisfy the simpler requirement that $\nu(A\cup B)=\max\{\nu(A),\nu(B)\}$. It is easy to check that $M_\cap X$ and $M_\cup X$ are closed subsets of $MX$.

The notion of density for an idempotent measure was introduced in \cite{A}.
For each possibility capacity $\nu\in M_\cup X$ we consider an upper semicontinuous function $[\nu]:X\to [0,1]$ that
sends each $x\in X$ to $\nu(x)$ and is called the density of $\nu$.
Observe that for a possibility capacity $\nu\in M_\cup X$ and a closed set $F\subset X$ we have $\nu(F) =
\max\{\nu(x) | x \in F\}$, and $\nu$ is completely determined by its values on singletons. It means that $\nu$ is completely determined by the function $[\nu]$.
 Conversely, each upper semicontinuous function $f:X\to I$ with $\max f = 1$ determines a possibility capacity $(f)\in M_\cup X$
by the formula $(f)(F) = \max\{f(x) | x \in F\}$, for a closed subset $F$ of $X$. It is easy to check that $([\nu])=\nu$ for each $\nu\in M_\cup X$ and $[(f)] = f$ for each upper semicontinuous function $f:X\to I$ with $\max f = 1$.

Denote $\varphi_t=\varphi^{-1}([t,+\infty))$ for each $\varphi\in C(X,[0,+\infty))$ and $t\in[0,+\infty)$.
Let us remind definitions of  the Choquet integral and the Sugeno integral with respect to a capacity $\mu\in MX$.  We consider for a compactum $X$ and  for a  function $f\in  C(X,[0,+\infty))$ an integral defined by the formula
$\int_X^{Ch} fd\mu=\int_0^\infty\mu(f_t)dt$ \cite{Cho}  and call it the Choquet integral.

For a  function $f\in  C(X,[0,1])$ we consider an integral defined by the formula $\int_X^{Sug} fd\mu=\max\{\mu(f_t)\wedge t\mid t\in[0,1]\}$ \cite{Su}  and call it the Sugeno integral. The existence of maximum follows from the semicontinuity of the capacity $\mu$.

Let us remark that the operation of minimum $\wedge$ is an important example of triangular norm (t-norm). Remind that triangular norm $\ast$ is a binary operation on the closed unit interval $[0,1]$ which is associative, commutative, monotone and $s\ast 1=s$ for each  $s\in [0,1]$ \cite{PRP}. We consider only continuous t-norms in this paper. Integrals obtained by changing the operation  $\wedge$ in the definition of Sugeno integral by any t-norm are called t-normed integrals and were studied in \cite{We1}, \cite{We2} and \cite{Sua}. So, for a continuous t-norm $\ast$ and a  function $f\in  C(X,[0,1])$ the corresponding t-normed integral is defined by the formula
$\int_X^{\vee\ast} fd\mu=\max\{\mu(f_t)\ast t\mid t\in[0,1]\}$.

Let $X$ be a compactum.  We call two functions $\varphi$, $\psi\in C(X,[0,1])$ comonotone (or equiordered) if $(\varphi(x_1)-\varphi(x_2))\cdot(\psi(x_1)-\psi(x_2))\ge 0$ for each $x_1$, $x_2\in X$. Let us remark that a constant function is comonotone to any function $\psi\in C(X,[0,1])$.

\begin{lemma}\label{Comon} Let  $\varphi$, $\psi\in C(X,[0,1])$ be two comonotone functions. Then we have $\varphi_t\subset\psi_t$ or $\varphi_t\supset\psi_t$ for each $t\in[0,1]$.
\end{lemma}

\begin{proof} Suppose the contrary. Then there exists $t\in[0,1]$, $x\in \varphi_t\setminus\psi_t$ and $y\in\psi_t\setminus\varphi_t$. Then we have $(\varphi(x)-\varphi(y))\cdot(\psi(x)-\psi(y))<0$.
\end{proof}

For $A\in\F(X)$ put $\Upsilon_A=\{\varphi\in C(X,[0,1])\mid \varphi(a)=1$ for each $a\in A\}$. If $A=\emptyset$ we put $\Upsilon_A=C(X,[0,1])$.

\begin{lemma}\label{Comon1} Let  $\varphi\in C(X,[0,1])$, $t<\max \varphi$ and $\psi\in\Upsilon_{\varphi_t}$. Then there exists $\psi'\in\Upsilon_{\varphi_t}$ such that $\psi'\le\psi$ and $\psi'$ is comonotone with $\varphi$.
\end{lemma}

\begin{proof} If $\psi=1_X$ we can put $\psi'=\psi$. So, consider the case when $1>d_0=\min\psi$. Choose a sequence $(d_i)$ converging to $1$ and such that $d_0<d_1<d_2<\dots 1$. Put $a_i=\max\{\varphi(x)\mid\psi(x)\le d_i\}$. Evidently we have $a_i<t$.

Consider the case when $a_i\rightarrow a<t$. Choose a monotone homeomorphism $\alpha:[a,t]\to [d_0,1]$ and put
$$ \psi'(x)=\begin{cases}
d_0,&\varphi(x)\le a,\\
\alpha(\varphi(x)),&a\le\varphi(x)\le t,\\
1,&t\le\varphi(x).\end{cases}$$

In the case when $a_i\rightarrow t$ we can assume $a_0<a_1<a_2<\dots$. For each $i\ge 1$ choose a monotone homeomorphism $\alpha_i:[a_i,a_{i+1}]\to [d_{i-1},d_i]$ and put
$$ \psi'(x)=\begin{cases}
d_0,&\varphi(x)\le a_1,\\
\alpha_i(\varphi(x)),&a_i\le\varphi(x)\le a_{i+1},\\
1,&t\le\varphi(x).\end{cases}$$

It is a routine check that $\psi'$ is a function we are looking for.
\end{proof}

Consider a characterization theorem of Sugeno integral for functions and capacities on finite $X$ proved in \cite{CB}.
It is proved in \cite{CB} that for a finite compactum $X$ a non-negative  functional $\mu$ on  $C(X,[0,1])=[0,1]^X$ satisfies the conditions:

\begin{enumerate}
\item  $\mu(1_X)=1$;
\item  $\mu(\psi)\le\mu(\varphi)$ for each functions $\varphi$, $\psi\in C(X,[0,1])$ such that $\varphi\le\psi$;
\item  $\mu(\psi\vee\varphi)=\mu(\psi)\vee\mu(\varphi)$ for each comonotone functions $\varphi$, $\psi\in C(X,[0,1])$;
\end{enumerate}

(4)$^\wedge$  $\mu(c_X\wedge\varphi)=c\wedge\mu(\varphi)$ for each $c\in [0,1]$, $\varphi\in C(X,[0,1])$,

if and only if there exists a unique capacity $\nu$ such that   $\mu$ is the Sugeno integral with respect
to $\nu$.

The analogous characterization theorem was  proved  in \cite{CLM} for each t-normed integral on a finite compactum. The authors also posed the problem to extend above mentioned results to any (infinite) compacta. We will consider such generalization in this section.

We consider any compactum $X$ and a continuous t-norm $\ast$.

\begin{lemma}\label{Bound} We have $\min_{x\in X}f(x)\le\int_X^{\vee\ast} fd\nu\le\max_{x\in X}f(x)$ for each capacity $\nu\in MX$ and $f\in  C(X,[0,1])$.
\end{lemma}

\begin{proof} Put $a=\min_{x\in X}f(x)$ and  $b=\max_{x\in X}f(x)$. The we have  $$a=1\ast a=\nu(f_a)\ast a\le\max\{\nu(f_t)\ast t\mid t\in[0,1]\}=$$
$$=\max\{\nu(f_t)\ast t\mid t\in[0,b]\}\le1\ast b=b.$$
\end{proof}

We denote by $\B$ the set of functionals $\mu:C(X,[0,1])\to[0,1]$ which satisfy the conditions:

\begin{enumerate}
\item $\mu(1_X)=1$;
\item $\mu(\varphi)\le\mu(\psi)$ for each functions $\varphi$, $\psi\in C(X,[0,1])$ such that $\varphi\le\psi$;
\item $\mu(\psi\vee\varphi)=\mu(\psi)\vee\mu(\varphi)$ for each comonotone functions $\varphi$, $\psi\in C(X,[0,1])$;
\item $\mu(c_X\ast\varphi)=c\ast\mu(\varphi)$ for each $c\in\R$ and $\varphi\in C(X,[0,1])$.

\end{enumerate}



Let us remark that for each $c\in[0,1]$ and for each $\mu\in\B$ the equality $\mu(c_X)=c$ follows from Properties 1 and 4.

\begin{theorem}\label{repr} A functional $\mu:C(X,[0,1])\to[0,1]$ is in $\B$
if and only if there exists a unique capacity $\nu$ such that   $\mu$ is the t-normed integral with respect
to $\nu$.
\end{theorem}

\begin{proof} Sufficiency. Consider any capacity $\nu\in MX$. By  $\mu$ we denote the  the t-normed integral with respect
to $\nu$. Then $\mu$ satisfies Property 1 by Lemma \ref{Bound}.
Consider any functions $\varphi$, $\psi\in C(X)$ such that $\varphi\le\psi$. The inequality $\mu(\varphi)\le\mu(\psi)$ follows from the obvious inclusion $\varphi_t\subset\psi_t$ and monotonicity of t-norm.

Let  $\varphi$, $\psi\in C(X,[0,1])$ be two comonotone functions. The inequality $\mu(\psi\vee\varphi)\ge\mu(\psi)\vee\mu(\varphi)$ follows from Property 2 proved above. We have $\nu(\psi_t)\ast t\le \mu(\psi)\vee\mu(\varphi)$  and $\nu(\varphi_t)\ast t\le \mu(\psi)\vee\mu(\varphi)$ for each $t\in [0,1]$. Lemma \ref{Comon} yields that $(\psi\vee\varphi)_t=\psi_t$ or $(\psi\vee\varphi)_t=\varphi_t$. Hence $\mu(\psi\vee\varphi)\le\mu(\psi)\vee\mu(\varphi)$ and we proved Property 3.

Consider any $c\in\R$ and $\psi\in C(X)$. Consider any $t\in[0,c]$ and put $b_t=\inf\{l\in[0,1]\mid t\le c\ast l\}$. It follows from continuity of $\ast$ that $c\ast b_t=t$. Moreover, we have $c\ast k\ge t$ iff $k\ge b_t$ for each $k\in [0,1]$. Since  $(c\ast\psi)_t=\emptyset$ for each $t>c$, we have $\mu(c\ast\psi)=\max\{\nu((c\ast\psi)_t)\ast t\mid t\in[0,c]\}=\max\{\nu(\psi^{-1}([b_t,1]))\ast b_t\ast c\mid t\in[0,c]\}\le\max\{\nu(\psi_s)\ast s\mid s\in[0,1]\}\ast c=c\ast\mu(\psi)$.

Choose $t_0\in [0,1]$ such that  $\mu(\psi)=\nu(\psi^{-1}([t_0,1]))\ast t_0$. Since $\ast$ is monotone, we have $(c\ast\psi)_{c\ast t_0}\supset\psi_{t_0}$ and $\nu(c\ast\psi)_{c\ast t_0})\ast c\ast t_0 \ge\nu(\psi_{t_0})\ast t_0\ast c=\mu(\psi)\ast c$. We  proved Property 4. Hence $\mu\in\B$.

Necessity. Take any $\mu\in\B$.  Define $\nu:\F(X)\to [0,1]$ as follows $\nu(A)=\inf\{\mu(\varphi)\mid \varphi\in \Upsilon_A\}$ if $A\ne\emptyset$ and $\nu(\emptyset)=0$. It is easy to see that $\nu$ satisfies Conditions 1 and 2 from the definition of capacity.

Let $\nu(A)<\eta$ for some $\eta\in [0,1]$ and $A\in\F(X)$. Then there exists $\varphi\in \Upsilon_A$ such that $\mu(\varphi)<\eta$. Choose $\beta\in I$ such that $\mu(\varphi)<\beta<\eta$. Since the operation $\ast$ is continuous,  $\eta\ast 1=\eta$ and $\eta\ast 0=0$, there is $\delta\in I$ such that $\eta\ast \delta=\beta$. Evidently, $\delta<1$. Choose $\zeta\in I$ such that $\delta<\zeta<1$ and a function $\psi\in\Upsilon_{\varphi_\zeta}$ such that  $\psi|_{\varphi^{-1}([0,\delta])}=\varphi|_{\varphi^{-1}([0,\delta])}$. Then we have $\delta\ast\psi\le\varphi$ and $\delta\ast\mu(\psi)\le\mu(\varphi)<b=\delta\ast\eta$. Hence $\mu(\psi)<\eta$. Put $U=\psi^{-1}((\zeta,1])$. Evidently $U$ is open and $U\supset A$. We have $\nu(K)\le\mu(\psi)<\eta$ for each compactum $K\subset U$. Hence $\nu\in MX$.

Let us show that $\int_X^{\vee\ast} \varphi d\nu=\mu(\varphi)$ for each $\varphi\in C(X,I)$.  We have $\int_X^{\vee\ast} \varphi d\nu=\max\{\inf\{\mu(\chi)\mid \chi\in \Upsilon_{\varphi_t}\}\ast t\mid t\in[0,1]\}=\max\{\inf\{\mu(t\ast\chi)\mid \chi\in \Upsilon_{\varphi_t}\}\mid t\in[0,1]\}$.

The inequality $\inf\{\mu(\chi)\mid \chi\in \Upsilon_{\varphi_t}\}\ast t\le\mu(\varphi)$ is obvious for each $t\in[0,\mu(\varphi)]$. Consider any $t>\mu(\varphi)$.  For each $\delta<t$ choose a function $\chi^\delta\in\Upsilon_{\varphi_t}$ such that $\chi^\delta|_{\varphi^{-1}([0,\delta])}=\varphi|_{\varphi^{-1}([0,\delta])}$. Then we have $\delta\ast\chi^\delta\le\varphi$ and $\delta\ast\mu(\chi^\delta)\le\mu(\varphi)$.  Since the operation $\ast$ is continuous, $\inf\{\mu(\chi)\mid \chi\in \Upsilon_{\varphi_t}\}\ast t\le\mu(\varphi)$. Hence $\int_X^{\vee\ast} \varphi d\nu\le\mu(\varphi)$.

Suppose $b=\int_X^{\vee\ast} \varphi d\nu<\mu(\varphi)=a$. Put $m=\max_{x\in X}\varphi(x)$.  Then for each $t\in[0,m]$ there exists $\chi^t\in\Upsilon_{\varphi_t}$ such that $\mu(t\ast\chi^t)<a$.  We can assume that $\chi^t$ is comonotone with $\varphi$ by Lemma \ref{Comon1}. For each $t\in[0,m)$ choose $t'>t$ such that $t'\ast\mu(\chi^t)<a$. The set $V_t=\{y\mid t'\ast\chi^t(y)>\varphi(y)\}$ is an open neighborhood for each $x\in X$ with $\varphi(x)=t$.

Now we will choose an open neighborhood $W$ of the set $\varphi_m$.
Put $\xi=\min\{\eta\in[0,1]\mid\eta\ast m\le m\}$. We have $\xi\ast m=m$. Since $a\le m$, using arguments as before we can find $\gamma\in[0,1]$ such that $\gamma\ast m=a$. Then we have $\xi\ast a=\xi\ast \gamma\ast m=\gamma\ast m=a$. Choose $\lambda<\xi$ such that $b<\lambda\ast a$. Then there exists $\omega\in\Upsilon_{\varphi_{m\ast\lambda}}$ such that $m\ast\lambda\ast\mu(\omega)<\lambda\ast a$, hence  $m\ast\mu(\omega)< a$. We can assume that $\omega$ is comonotone with $\varphi$ by Lemma \ref{Comon1}. Since $m\ast\lambda<m$, the open set $W=\varphi^{-1}(m\ast\lambda,m]$ contains the set $\varphi_m$. Let us remark that $m\ast\omega(x)= m$ for each $x\in W$.

  We can choose a finite subcover $\{W,V_{t_1},\dots,V_{t_k}\}$ of the open cover $\{W\}\cup\{V_t\mid t\le m\}$. We can assume that $t_0=m\ast\lambda$ and $t_i\in [0,m)\setminus \{t_0\}$ for $i>0$.
 Then we have that the functions $t_i'\ast\chi^{t_i}$ and $\omega$ are pairwise comonotone, hence $\mu(\omega\bigvee(\bigvee_{i=1}^k\{t_i'\ast\chi^{t_i}\}))<a$. On the other hand $\varphi\le \omega\bigvee(\bigvee_{i=1}^k\{t_i'\ast\chi^{t_i}\})$ and we obtain a contradiction.
\end{proof}

Let us remark that for Sugeno integral (when $\ast=\min$) instead Property 3 we can consider a weaker condition:
 $\mu(c_X\bigvee\varphi)=c\ast\mu(\varphi)$ for each $c\in\R$ and $\varphi\in C(X,[0,1])$ see \cite{Nyk} and \cite{R}.

For $\psi\in C(X,[0,1])$ we define a function $l_X^\psi:MX\to[0,1]$ by the formula $l_X^\psi(\nu)=\int_X^{\vee\ast} \psi d\nu$. We also define a map $l_X:MX\to[0,1]^{C(X,[0,1])}$ taking the diagonal product $l_X=(l_X^\psi)_{\psi\in C(X,[0,1])}$.

\begin{lemma}\label{contint} The map $l_X^\psi$ is continuous for each $\psi\in C(X,[0,1])$.
\end{lemma}

\begin{proof} Consider any $\nu\in MX$ such that $l_X^\psi(\nu)<a$ for some $a\in(0,1]$. Put $\varepsilon=a-l_X^\psi(\nu)$. Since the map $\ast:[0,1]\times[0,1]\to[0,1]$ is continuous and the space $[0,1]\times[0,1]$ is compact, there exists $\delta>0$ such that for each $(r_1,r_2)$, $(p_1,p_2)\in [0,1]\times[0,1]$ such that $|r_1-p_1|<\delta$ and $|r_2-p_2|<\delta$ we have $|r_1\ast r_2-p_1\ast p_2|<\varepsilon$. Choose $k\in\N$ such that $\frac{1}{k}<\delta$ and put $t_i=\frac{i}{k}$ for $i\in\{0,\dots,k\}$. Define an open set $O_i=\{\mu\in MX\mid \mu(\psi_{t_i})<\nu(\psi_{t_i})+\delta\}$ and put $O=\cap_{i=1}^kO_i$. Evidently $O$ is an open neighborhood of $\nu$. Consider any $\mu\in O$ and $t\in[0,1]$. Let $i$ be a maximal element of $\{0,\dots,k\}$ such that $t_i\le t$. Then we have $\mu(\psi_{t})\ast t\le\mu(\psi_{t_i})\ast t<\nu(\psi_{t_i})\ast t_i+\varepsilon\le a$. Hence $l_X^\psi(\mu)<a$.

Now, consider any $\nu\in MX$ such that $l_X^\psi(\nu)>a$ for some $a\in[0,1)$. Then there exists $t\in[0,1]$ such that $\nu(\psi_{t})\ast t> a$. Put $\varepsilon=\nu(\psi_{t})\ast t- a$. As before we choose $\delta>0$ such that for each $(r_1,r_2)$, $(p_1,p_2)\in [0,1]\times[0,1]$ such that $|r_1-p_1|<\delta$ and $|r_2-p_2|<\delta$ we have $|r_1\ast r_2-p_1\ast p_2|<\varepsilon$.  Define an open set $O=\{\mu\in MX\mid \mu(\psi^{-1}(t-\delta,1])>\nu(\psi^{-1}(t-\delta,1])-\delta\}$. Evidently $O$ is an open neighborhood of $\nu$. Consider any $\mu\in O$. There exists $p\in(t-\delta,t]$ such that $\mu(\psi_p)>\nu(\psi^{-1}(t-\delta,1])-\delta\ge\nu(\psi_{t})-\delta$. Then we have $\mu(\psi_{p})\ast p> \nu(\psi_{t})\ast t-\varepsilon=a$. Hence $l_X^\psi(\mu)>a$ and the map $l_X^\psi$ is continuous.
\end{proof}

\begin{corollary}\label{emb} The map $l_X$ is a topological embedding.
\end{corollary}

\section{Tensor products of capacities} For a continuous map of compacta $f:X\to Y$ we define the map $Mf:MX\to MY$ by the formula $Mf(\nu)(A)=\nu(f^{-1}(A))$ where $\nu\in MX$ and $A\in\F(Y)$. The map $Mf$ is continuous.  In fact, this extension of the construction $M$ defines the capacity functor in the category of compacta and continuous maps. The categorical technics are very useful for investigation of capacities on compacta (see \cite{NZ} for more details). We try to avoid the formalism of category theory in this paper,  but we follow the main ideas of such approach.

The tensor product operation of probability measures is well known and very useful partially for investigation of the spaces of probability measures on compacta (see for example Chapter 8 from \cite{FF}). General categorical definition of tensor product for any functor was given in \cite{BR}. Applying this definition to the capacity functor we obtain that a tensor product of capacities on compacta $X_1$ and $X_2$ is  a continuous map $$\otimes:MX_1\times MX_2\to M(X_1\times X_2)$$ such that for each $i\in\{1,2\}$ we have $M(p_i)\circ\otimes= \pr_i$ where $p_i:X_1\times X_2\to X_i$, $\pr_i:MX_1\times MX_2\to MX_i$ are the corresponding projections.

A tensor product for capacities was introduced in \cite{KZ}.  This definition is based on the capacity monad structure.   An explicit formula for evaluating  tensor product of capacities was given in \cite{R4} omitting the formalism of category theory. For $\mu_1\in MX_1$, $\mu_2\in MX_2$ and $B\in\F(X_1\times X_2)$ we put $$\mu_1\otimes\mu_2(B)=\sup\{t\in[0,1]\mid\mu_1(\{x\in X_1\mid \mu_2(p_2((\{x\}\times X_2)\cap B))\ge t\}\ge t\}.$$

The problem of multiplication of capacities was deeply considered in the possibility theory and it application to the game theory and the decision making theory where the term joint possibility distribution is used. A standard choice  of a joint possibility distribution is based on the minimum operation. For $\mu_1\in M_\cup X_1$, $\mu_2\in M_\cup X_2$ and $(x,y)\in X_1\times X_2$ we put $$[\mu_1\otimes\mu_2](x,y)=[\mu_1](x)\wedge[\mu_2](y).$$ (Let us remind that by $[\nu]$ we denote the density of a possibility capacity $\nu$.) It is easy to check that both definitions coincide in the class of possibility capacities, the difference is only in terms.

 A more general approach is also used where the minimum operation is changed by any t-norm (see for example \cite{DP2}). We will use this definition in our paper but we  prefer the term tensor product.

So, we fix a continuous t-norm $\ast$ and consider a tensor product generated by $\ast$ defined as follows. For possibility capacities $\mu_1\in M_\cup X_1$, $\mu_2\in M_\cup X_2$ and $(x,y)\in X_1\times X_2$ we put $$[\mu_1\circledast\mu_2](x,y)=[\mu_1](x)\ast[\mu_2](y).$$

We also can generalize the above mentioned formula from \cite{R4}. For capacities $\mu_1\in MX_1$, $\mu_2\in MX_2$ and $B\in\F(X_1\times X_2)$ we put $$\mu_1\widetilde{\circledast}\mu_2(B)=\sup\{t\in[0,1]\mid\mu_1(\{x\in X_1\mid \mu_2(p_2((\{x\}\times X_2)\cap B))\ge t\})\ast t\}.$$

The following theorem shows that  both definitions coincide in the class of possibility capacities.

\begin{theorem} For possibility capacities $\mu_1\in M_\cup X_1$, $\mu_2\in M_\cup X_2$ and $(x,y)\in X_1\times X_2$ we have $$[\mu_1\widetilde{\circledast}\mu_2](x,y)=[\mu_1](x)\ast[\mu_2](y).$$
\end{theorem}

\begin{proof} We have $$p_2((\{z\}\times X_2)\cap \{(x,y)\}))=\begin{cases}
\emptyset,&z\neq x,\\
\{y\},&z=x\end{cases},$$
thus
$$\{z\in X_1\mid \mu_2(p_2((\{z\}\times X_2)\cap \{(x,y)\}))\ge t\})=\begin{cases}
\{x\},&t\le\mu_2(\{y\}),\\
\emptyset,&t>\mu_2(\{y\})\end{cases}.$$
Hence $[\mu_1\widetilde{\circledast}\mu_2](x,y)=\mu_1\widetilde{\circledast}\mu_2(\{(x,y)\}))
=\mu_1(\{x\})\ast\mu_2(\{y\})=[\mu_1](x)\ast[\mu_2](y).$

\end{proof}

It was noticed in \cite{KZ} that we can extend the definition of tensor product to any finite number of factors by induction.

\section{Nash equilibrium in mixed strategies} Let us recall the notion of Nash equilibrium and some facts concerning existence of such equilibrium. We consider an $n$-players game $f:X=\prod_{j=1}^n X_j\to\R^n$ with compact Hausdorff spaces of strategies $X_i$. The coordinate function $f_i:X\to \R$ is called the payoff function of $i$-th player. For $x\in X$ and $t_i\in X_i$ we use the notation $(x;t_i)=(x_1,\dots,x_{i-1},t_i,x_{i+1},\dots,x_n)$. A point $x\in X$ is called a Nash max-equilibrium (min-equilibrium)  point if for each $i\in\{1,\dots,n\}$ and for each $t_i\in X_i$ we have $f_i(x;t_i)\le f_i(x)$ ($f_i(x;t_i)\ge f_i(x)$).

Usually some additional convexity structures are needed to establish existence of  Nash equilibrium. A family $\C$ of closed subsets of a compactum $X$ is
called a {\it convexity} on $X$ if $\C$ is stable for intersection
and contains $X$ and the empty set. The elements of $\C$ are called
$\C$-convex (or simply convex). Although we follow general concept  of abstract convexity from \cite{vV}, our definition is different.
We consider only closed convex sets. Such structure is called a closure structure in \cite{vV}. Our definition is the same as in \cite{W}. The whole family of convex sets in the sense of \cite{vV} could be obtained by the operation of
union of up-directed families. In what follows, we assume that each convexity contains all singletons.

A convexity $\C$ on $X$ is called $T_4$ (normal) if for each disjoint $C_1$, $C_2\in
\C$ there exist $S_1$, $S_2\in\C$ such that $S_1\cup S_2=X$,
$C_1\cap S_2=\emptyset$ and $C_2\cap S_1=\emptyset$ (see for example \cite{RZ}).

Now, let $\C_i$ be a convexity on $X_i$. We say that a function $f_i:X=\prod_{i=1}^n X_i\to\R$ is quasiconcave (quasiconvex) with respect to the $i$-th variable if we have $(f_i^x)^{-1}([t;+\infty))\in\C_i$ ($(f_i^x)^{-1}((-\infty;t])\in\C_i$) for each $t\in\R$ and $x\in X$ where $f_i^x:X_i\to\R$ is the function defined as follows $f_i^x(t_i)=f_i(x;t_i)$ for $t_i\in X_i$.

\begin{theorem}\label{NN}\cite{R3} Let $f:X=\prod_{j=1}^n X_j\to\R^n$ be a game with a  normal convexity  $\C_i$ defined  on each compactum $X_i$ such that all convex sets are connected, the function $f$ is continuous  and the function $f_i:X\to\R$ is quasiconcave (quasiconvex) with respect to the $i$-th variable for each $i\in\{1,\dots,n\}$. Then there exists a Nash max-equilibrium (min-equilibrium) point.
\end{theorem}

Let us remark that the previous theorem was proved in \cite{R3} only for the max-equilibrium. But the proof is the same for the min-equilibrium.

Now we apply these general concepts to the spaces of possibility capacities. We consider a game $u:Z=\prod_{i=1}^n Z_i\to[0,1]^n$ with compact Hausdorff spaces of pure strategies $Z_1,\dots,Z_n$ and continuous payoff functions $u_i:\prod_{i=1}^n Z_i\to[0,1]$. Let  $\star$ and $\ast$ be two t-norms.  We will extend the game $u:Z=\prod_{i=1}^n Z_i\to[0,1]^n$ to a game in mixed strategies  $eu:\prod_{i=1}^n M_\cup Z_i\to[0,1]^n$ using the integral generated by t-norm $\star$ and the tensor product generated by t-norm $\ast$.

 We define expected payoff functions $eu_i:\prod_{j=1}^n M_\cup Z_j\to[0,1]$ by the formula  $$eu_i(\nu_1,\dots,\nu_n)=\int_X^{\vee\star} u_i d(\nu_1\circledast\dots\circledast\nu_n)$$
 for $(\nu_1,\dots,\nu_n)\in\prod_{j=1}^n M_\cup Z_j$.

Lemma \ref{contint} and continuity of tensor product imply the following lemma.

 \begin{lemma}\label{cont} The function $eu_i$ is continuous for each $i\in\{1,\dots,n\}$.
\end{lemma}

We discuss existence of Nash equilibrium in mixed strategies represented by possibility capacities. There exist a trivial solution of the problem for max-equilibrium. We can consider the natural order on $M_\cup Z_i$. Then each $M_\cup Z_i$ contains the greatest element $\mu_i$ defined by the formula $$\mu_i(A)=\begin{cases}
0,&A=\emptyset,\\
1,&A\neq\emptyset\end{cases}$$ for $A\in\F(Z_i)$. Hence $(\mu_1,\dots,\mu_n)$ is a Nash max-equilibrium point.  There is no such trivial solution for the min-equilibrium, since $M_\cup Z_i$ does not contain the smallest element.

We will need some convexity structure on $M_\cup X$ to establish existence of the min-equilibrium. We use an idempotent convexity considered in \cite{Ch} and \cite{BCR} for finite-dimensional spaces where it was called B-convexity. Firstly, we introduce it on a cube  $[0,1]^S$, where $S$ is any set (finite or infinite).  We call a subset $C$ of $[0,1]^S$ B-convex if for each $x$, $y\in [0,1]^S$ and  $\alpha\in[0,1]$ we have $\alpha\cdot x\vee y\in C$ (the operations of maximum $\vee$ and multiplication for a scalar $\cdot$ are taken coordinate-wise).

Partially, we can consider B-convexity on $M_\cup X$ for each compactum $X$. Take any $\nu$, $\mu\in M_\cup X$ and $s\in[0,1]$. Put $(s\cdot\nu\vee\mu)(A)=s\cdot\nu(A)\vee\mu(A)$ for $A\in\F(X)$. It is easy to check that $s\cdot\nu\vee\mu\in M_\cup X$. It is also easy to see that the introduced operation commutes with taking the density, i.e. $[s\cdot\nu\vee\mu]=s\cdot[\nu]\vee[\mu]$ (we consider $[s\cdot\nu\vee\mu]$, $[\nu]$ and $[\mu]$ as elements of $[0,1]^X$.  We call a subset $C$ of $M_\cup X$ B-convex if for each $\nu$, $\mu\in M_\cup X$ and  $s\in[0,1]$ we have $s\cdot\nu\vee\mu\in C$. Evidently, each B-convex set is connected. So, we consider on $M_\cup X$ a convexity structure $\C_X$ which consists of all closed B-convex subsets of  $M_\cup X$.

The proof of the following lemma reduces to routine checking and so we omit it.

\begin{lemma}\label{CPC} The set $l_X(A)$ is B-convex in $[0,1]^{C(X,[0,1])}$ for each $A\in \C_X$ and $l_X^{-1}(B)\in\C_X$ for each closed B-convex subset $B\subset[0,1]^{C(X,[0,1])}$.
\end{lemma}

\begin{lemma}\label{BN} The convexity $\C_X$ is normal for each compactum $X$.
\end{lemma}

\begin{proof} Let $A$ and $D$ be two  B-convex disjoint closed subsets of $M_\cup X$. Then $l_X(A)$ and $l_X(D)$ are two  B-convex disjoint closed subsets of $[0,1]^{C(X,[0,1])}$ by Lemma \ref{CPC} and Corollary \ref{emb}. It follows from compactness of $l_X(A)$ and $l_X(D)$ and properties of the product topology on $[0,1]^{C(X,[0,1])}$ that there exists a finite subset $N$ of $C(X,[0,1])$ such that $p_N(l_X(A))\cap p_N(l_X(D))=\emptyset$ where $p_N:[0,1]^{C(X,[0,1])}\to[0,1]^N$ is the natural projection. Evidently $p_N(l_X(A))$ and $p_N(l_X(D))$ are B-convex disjoint compact subsets of $\R^N$.

Theorem 7.1 from \cite{BCR}  implies that there exist two B-convex closed subsets $L_1$, $L_2$ of $[0,1]^N$ such that $L_1\cup L_2=[0,1]^N$ and $L_1\cap p_N(l_X(B))=\emptyset=L_2\cap p_N(l_X(A))$. Then $(p_N\circ l_X)^{-1}(L_1)$ and $(p_N\circ l_X)^{-1}(L_2)$ are B-convex closed subsets of $M_\cup X$ we are looking for.
\end{proof}

Let us remark that each t-norm is distributive respectively the maximum operation, i.e. $t\ast(s\vee l)=(t\ast s)\vee(t\ast l)$. It follows from the monotonicity property.

\begin{lemma}\label{QC} The map $eu_i:\prod_{j=1}^n M_\cup Z_j\to[0,1]$ is quasiconvex  with respect to the $i$-th variable for each $i\in\{1,\dots,n\}$.
\end{lemma}

\begin{proof} We will prove the lemma for the case $n=2$. The proof of the general case is the same. We also can assume $i=1$. Consider any $s\in[0,1]$ and $\mu\in M_\cup Z_2$. We should show that $(eu_1^\mu)^{-1}([0,s])$ is B-convex. Consider any capacities $\nu_1$, $\nu_2\in M_\cup Z_1$,  such that $eu_1(\nu_j,\mu)\le s$ for each $j\in\{1,2\}$.

Choose any $c\in[0,1]$.  Then we have $$eu_1(c\cdot\nu_1\vee\nu_2,\mu)=\int_X^{\vee\star} u_1 d((c\cdot\nu_1\vee\nu_2)\circledast\mu)=$$
$$=\max\{((c\cdot\nu_1\vee\nu_2)\circledast\mu)(u_1^{-1}([0,t]))\star t\mid t\in[0,1]\}=$$

\centerline {(we put $U_t=u_1^{-1}([0,t])$)}
$$=\max\{\max\{[(c\cdot\nu_1\vee\nu_2)\circledast\mu](x,y)\mid (x,y)\in U_t\}\star t\mid t\in[0,1]\}=$$
$$=\max\{\max\{c\cdot[\nu_1](x)\ast[\mu](y)\vee[\nu_2](x)\ast[\mu](y)\mid (x,y)\in U_t\}\star t\mid t\in[0,1]\}\le$$
$$\le\max\{\max\{[\nu_1](x)\ast[\mu](y)\mid (x,y)\in U_t\}\star t\mid t\in[0,1]\}\vee$$
$$\vee\max\{\max\{[\nu_2](x)\ast[\mu](y)\mid (x,y)\in U_t\}\star t\mid t\in[0,1]\}=$$
$$eu_1(\nu_1,\mu)\vee eu_1(\nu_2,\mu)\le s.$$

\end{proof}

Theorem \ref{NN} and Lemmas \ref{BN}, \ref{QC} imply the following theorem.

\begin{theorem}\label{NMP} There exists a Nash min-equilibrium  point for the game with the expected payoff functions $eu_i:\prod_{j=1}^n M_\cup Z_j\to[0,1]$.
\end{theorem}

We finish this section with an example of a game in possibility capacities with expected payoff function represented by  Choquet integral which has no Nash min-equilibrium. Consider the 2-person game in pure strategies $u:\{a,b\}\times\{a,b\}\to\R^2$, where  $u_1(a;a)=3$, $u_1(a;b)=0$, $u_1(b;a)=1$, $u_1(b;b)=2$ and $u_2(a;a)=0$, $u_2(a;b)=3$, $u_2(b;a)=2$, $u_2(b;b)=1$ and define the game in mixed strategies $cu:M_\cup (\{a,b\})\times M_\cup (\{a,b\})\to\R^2$ as follows $$cu_i(\nu_1;\nu_2)=\int_X^{Ch} u_i d(\nu_1\circledast\nu_2)$$
 for $(\nu_1,\nu_2)\in M_\cup (\{a,b\})\times M_\cup (\{a,b\})$ and $i\in\{1,2\}$, where $\circledast$ is the tensor product generated by the minimum operation $\wedge$. Let us show that such game has no Nash min-equilibrium  point.

 Consider any pair of mixed strategies $(\nu,\mu)\in M_\cup (\{a,b\})\times M_\cup (\{a,b\})$ with $[\nu](a)=\lambda_1$, $[\nu](b)=\beta_1$ and $[\mu](a)=\lambda_2$, $[\mu](b)=\beta_2$. Then we have $$cu_1(\nu,\mu)=\beta_1\wedge\lambda_2+2\beta_1\wedge\beta_2+3\lambda_1\wedge\lambda_2$$ and $$cu_2(\nu,\mu)=2\beta_1\wedge\lambda_2+\beta_1\wedge\beta_2+3\lambda_1\wedge\beta_2.$$

Since $\nu$, $\mu\in M_\cup (\{a,b\})$, we have $\max\{\lambda_1,\beta_1\}=1=\max\{\lambda_2,\beta_2\}$.

Consider the case $\lambda_1=1=\lambda_2$.  If $\beta_2=1$, we consider $\mu'\in M_\cup (\{a,b\})$ with $[\mu'](a)=1$, $[\mu'](b)=0$. Then we have $$cu_2(\nu,\mu')=2\beta_1<3+3\beta_1=cu_2(\nu,\mu).$$ If $\beta_2<1$, we consider $\nu'\in M_\cup (\{a,b\})$ with $[\nu'](a)=0$, $[\nu'](b)=1$. Then we have $$cu_1(\nu',\mu)= 1+2\beta_2<3\le cu_1(\nu,\mu).$$

Now, let  $\lambda_1=1=\beta_2$. Consider $\mu'\in M_\cup (\{a,b\})$ with $[\mu'](a)=1$, $[\mu'](b)=0$. Then we have $$cu_2(\nu,\mu')=2\beta_1<\beta_1+2\beta_1\wedge\lambda_2+3=cu_2(\nu,\mu).$$

Consider the case $\beta_1=1=\beta_2$.  If $\lambda_2=1$, we consider $\mu'\in M_\cup (\{a,b\})$ with $[\mu'](a)=0$, $[\mu'](b)=1$. Then we have $$cu_2(\nu,\mu')=1+3\lambda_1<3+3\lambda_1=cu_2(\nu,\mu).$$ If $\lambda_2<1$, we consider $\nu'\in M_\cup (\{a,b\})$ with $[\nu'](a)=1$, $[\nu'](b)=0$. Then we have $$cu_1(\nu',\mu)=3\lambda_2<\lambda_2+2+3\lambda_1\wedge\lambda_2=cu_1(\nu,\mu).$$

Finally,  let  $\lambda_2=1=\beta_1$. If $\lambda_1>0$, we consider $\nu'\in M_\cup (\{a,b\})$ with $[\nu'](a)=0$, $[\nu'](b)=1$. Then we have $$cu_1(\nu',\mu)=1+2\beta_2<1+2\beta_2+3\lambda_1=cu_1(\nu',\mu).$$ If $\lambda_1=0$, we consider $\mu'\in M_\cup (\{a,b\})$ with $[\mu'](a)=0$, $[\mu'](b)=1$. Then we have $$cu_2(\nu,\mu')=1<2+\beta_2=cu_2(\nu,\mu).$$

Hence $(\nu,\mu)$ is not a Nash min-equilibrium  point.

\section{Conclusion} We consider  games where players are allowed   to play their mixed non-additive strategies expressed by
possibility capacities. Such games with payoff functions expressed by  Choquet
integral and  Sugeno integral where considered in  \cite{KZ}, \cite{GM}, \cite{R3} and \cite{HM}.    Since the space of all capacities and the space of possibility capacities have the greatest  element, the existence  problem of Nash equilibrium there is rather trivial. But the set of possibility capacities has no smallest element. So,  we considered a game  where  the players try to minimize his expected payoff function represented by fuzzy integral generated by the maximum operation and some continuous triangular norm which  is a generalization of the  Sugeno integral. In Section 2 we give a characterization of such integrals for any compacta solving the problem posed in \cite{CLM} where such characterization was given for finite compacta. We also consider a generalization of tensor product of possibility capacities using any t-norm. (Tensor product considered in above cited papers was based on the minimum operation). In Section 4 we proved existence of Nash equilibrium for  considered games.   We also  provide an example showing that there is no Nash equilibrium when  expected payoff functions are represented by  Choquet integral.


\end{document}